\documentclass[english]{article}
\usepackage[latin9]{inputenc}
\usepackage{babel}
\usepackage{comment}
\usepackage{varioref}
\usepackage{amsmath}
\usepackage{amsthm}
\usepackage{amssymb}
\usepackage{stmaryrd}
\usepackage{esint}
\PassOptionsToPackage{normalem}{ulem}
\usepackage{ulem}
\usepackage{color}
\usepackage[all]{xy}
\usepackage[nottoc,numbib]{tocbibind}
\usepackage[CJKbookmarks=true]{hyperref}
\makeatletter

\newcommand\specialsectioning{\setcounter{secnumdepth}{-2}}
  
\DeclareFontEncoding{LGR}{}{}

\ProvideTextCommand{\~}{LGR}[1]{\char126#1}

\DeclareTextSymbolDefault{\textquotedbl}{T1}

\theoremstyle{plain}
\newtheorem{thm}{\protect\theoremname}[section]
\theoremstyle{plain}
\newtheorem*{thm*}{\protect\theoremname}
\theoremstyle{plain}
\newtheorem{ques}[thm]{\protect\questionname}
\theoremstyle{plain}
\newtheorem{lem}[thm]{\protect\lemmaname}
\theoremstyle{definition}

\theoremstyle{definition}

\theoremstyle{plain}

\theoremstyle{definition}
\newtheorem*{rem*}{\protect\remarkname}
\theoremstyle{plain}
\newtheorem{cor}[thm]{\protect\corollaryname}
\theoremstyle{plain}
\newtheorem*{fact*}{\protect\factname}
\theoremstyle{definition}

\theoremstyle{definition}

\usepackage{babel}
\providecommand{\corollaryname}{Corollary}
  \providecommand{\definitionname}{Definition}
  \providecommand{\factname}{Fact}
  \providecommand{\lemmaname}{Lemma}
  \providecommand{\propositionname}{Proposition}
  \providecommand{\remarkname}{Remark}
\providecommand{\theoremname}{Theorem}
\providecommand{\examplename}{Example}
\providecommand{\questionname}{Conjecture}

\newcommand\blfootnote[1]{ 
\begingroup 
\renewcommand\thefootnote{}\footnote{#1}
\addtocounter{footnote}{-1}
\endgroup 
}

\usepackage{babel}
\providecommand{\corollaryname}{Corollary}
  \providecommand{\definitionname}{Definition}
  \providecommand{\factname}{Fact}
  \providecommand{\lemmaname}{Lemma}
  \providecommand{\propositionname}{Proposition}
  \providecommand{\remarkname}{Remark}
\providecommand{\theoremname}{Theorem}

\usepackage{babel}
\providecommand{\corollaryname}{Corollary}
  \providecommand{\definitionname}{Definition}
  \providecommand{\factname}{Fact}
  \providecommand{\lemmaname}{Lemma}
  \providecommand{\propositionname}{Proposition}
  \providecommand{\remarkname}{Remark}
\providecommand{\theoremname}{Theorem}

\usepackage{babel}
\providecommand{\corollaryname}{Corollary}
  \providecommand{\definitionname}{Definition}
  \providecommand{\factname}{Fact}
  \providecommand{\lemmaname}{Lemma}
  \providecommand{\propositionname}{Proposition}
  \providecommand{\remarkname}{Remark}
\providecommand{\theoremname}{Theorem}

\usepackage{babel}
\providecommand{\corollaryname}{Corollary}
  \providecommand{\definitionname}{Definition}
  \providecommand{\factname}{Fact}
  \providecommand{\lemmaname}{Lemma}
  \providecommand{\propositionname}{Proposition}
  \providecommand{\remarkname}{Remark}
\providecommand{\theoremname}{Theorem}

\usepackage{babel}
\providecommand{\corollaryname}{Corollary}
  \providecommand{\definitionname}{Definition}
  \providecommand{\factname}{Fact}
  \providecommand{\lemmaname}{Lemma}
  \providecommand{\propositionname}{Proposition}
  \providecommand{\remarkname}{Remark}
\providecommand{\theoremname}{Theorem}

\usepackage{babel}
\providecommand{\corollaryname}{Corollary}
  \providecommand{\definitionname}{Definition}
  \providecommand{\factname}{Fact}
  \providecommand{\lemmaname}{Lemma}
  \providecommand{\propositionname}{Proposition}
  \providecommand{\remarkname}{Remark}
\providecommand{\theoremname}{Theorem}

\usepackage{babel}
\providecommand{\corollaryname}{Corollary}
  \providecommand{\definitionname}{Definition}
  \providecommand{\factname}{Fact}
  \providecommand{\lemmaname}{Lemma}
  \providecommand{\propositionname}{Proposition}
  \providecommand{\remarkname}{Remark}
\providecommand{\theoremname}{Theorem}

\usepackage{babel}
\providecommand{\corollaryname}{Corollary}
  \providecommand{\definitionname}{Definition}
  \providecommand{\factname}{Fact}
  \providecommand{\lemmaname}{Lemma}
  \providecommand{\propositionname}{Proposition}
  \providecommand{\remarkname}{Remark}
\providecommand{\theoremname}{Theorem}

\usepackage{babel}
\providecommand{\corollaryname}{Corollary}
  \providecommand{\definitionname}{Definition}
  \providecommand{\factname}{Fact}
  \providecommand{\lemmaname}{Lemma}
  \providecommand{\propositionname}{Proposition}
  \providecommand{\remarkname}{Remark}
\providecommand{\theoremname}{Theorem}

\usepackage{babel}
\providecommand{\corollaryname}{Corollary}
  \providecommand{\definitionname}{Definition}
  \providecommand{\factname}{Fact}
  \providecommand{\lemmaname}{Lemma}
  \providecommand{\propositionname}{Proposition}
  \providecommand{\remarkname}{Remark}
\providecommand{\theoremname}{Theorem}

\usepackage{babel}
\providecommand{\corollaryname}{Corollary}
  \providecommand{\definitionname}{Definition}
  \providecommand{\factname}{Fact}
  \providecommand{\lemmaname}{Lemma}
  \providecommand{\propositionname}{Proposition}
  \providecommand{\remarkname}{Remark}
\providecommand{\theoremname}{Theorem}

\usepackage{babel}
\providecommand{\corollaryname}{Corollary}
  \providecommand{\definitionname}{Definition}
  \providecommand{\factname}{Fact}
  \providecommand{\lemmaname}{Lemma}
  \providecommand{\propositionname}{Proposition}
  \providecommand{\remarkname}{Remark}
\providecommand{\theoremname}{Theorem}

\makeatother

\providecommand{\corollaryname}{Corollary}
\providecommand{\definitionname}{Definition}
\providecommand{\factname}{Fact}
\providecommand{\lemmaname}{Lemma}
\providecommand{\propositionname}{Proposition}
\providecommand{\remarkname}{Remark}
\providecommand{\theoremname}{Theorem}

\newcommand{\Addresses}{{
  \bigskip
  \footnotesize

  \textsc{ School of Mathematical Sciences, Xiamen University, Xiamen, Fujian 361005, P.R. China;}
  
 Zhengxing Lian: \texttt{lianzx@mail.ustc.edu.cn; lianzx@xmu.edu.cn}

\medskip  
  
  \par\nopagebreak

 \textsc{Institute of Mathematics, Polish Academy of Sciences, ul. \'{S}niadeckich 8, 00-656 Warszawa, Poland.}

  Ruxi Shi: \texttt{rshi@impan.pl}

}}
\title{A counter-example for polynomial version of Sarnak's conjecture}
\author{Zhengxing Lian, Ruxi Shi
}

\begin{document}
\maketitle
\blfootnote{Z.L. was partially supported by the National Science Centre (Poland) grant 2016/22/E/ST1/00448 and Xiamen Youth Innovation Foundation No. 3502Z20206037.}
\begin{abstract}
    We construct the counter-example for polynomial version of Sarnak's conjecture for minimal systems, which assets that the M\"obius function is linearly disjoint from subsequences along polynomials of deterministic sequences realized in minimal systems. Our example is in the class of Toeplitz systems, which are minimal.
\end{abstract}

\section{Introduction}\label{sec:1}
The M\"{o}bius function $\mu: \mathbb{N}\rightarrow \{-1,0,1\}$ is defined by
$\mu(1)=1$ and
\begin{equation}\label{M-function}
  \mu(n)=\left\{
           \begin{array}{ll}
             (-1)^k & \hbox{if $n$ is a product of $k$ distinct primes;} \\
             0 & \hbox{otherwise.}
           \end{array}
         \right.
\end{equation}
Let $(X,T)$ be a {\it topological dynamical system}, namely $X$ is a compact metrizable space and $T:X\to X$ a homeomorphism. We say that a sequence $\xi=(\xi(n))_{n\in \mathbb{N}}$ is {\em realized} in $(X,T)$ if there is a continuous function $f\in C(X)$ on $X$ and a point $x\in X$
such that $\xi(n) = f(T^nx)$ for all $n\in\mathbb{N}$. Moreover, the sequence $\xi$ is said to be {\em deterministic} if it is realized in a
system of zero topological entropy. Here is the well-known Sarnak conjecture  \cite{Sar}.

\medskip

\noindent {\bf Sarnak Conjecture:}\ {\em
The M\"{o}bius function $\mu$ is linearly disjoint from any deterministic sequence $\xi$. That is,
\begin{equation}\label{Sarnak}
  \lim_{N\rightarrow \infty}\frac{1}{N}\sum_{n=1}^N\mu(n)\xi(n)=0.
\end{equation}
}
Eisner proposed a variant of Sarnak conjecture, called polynomial Sarnak conjecture for minimal systems \cite[Conjecture 2.3]{E2018}.

\medskip

\noindent {\bf Polynomial Sarnak Conjecture:}\ {\em
Let $(X,T)$ be a minimal topological dynamical system with entropy zero. Then 
\begin{equation}\label{polynomial sar eq}
  \lim_{N\rightarrow \infty}\frac{1}{N}\sum_{n=1}^N\mu(n)f(T^{p(n)}x)=0
\end{equation}
holds for every $f\in C(X)$, every polynomial $p:\mathbb{N}\rightarrow \mathbb{N}_0$ and every $x\in X$.
}

\medskip

Obviously, polynomial Sarnak conjecture implies Sarnak conjecture. It is well-known that Sarnak conjecture follows from Chowla conjecture (see \cite{Sar,AbKuLeDe,Tao}). Even though Sarnak conjecture has been proven for various dynamical systems, it is still widely open in general. We refer to \cite{FKL2018} for a survey of many recent results on Sarnak conjecture. See also \cite{AbKuLeDe,huang2017sequences,Shi2018} for other research around Sarnak conjecture.

Green and Tao \cite[Theorem 1.1]{GreenTao} proved that nilsystems satisfy polynomial Sarnak conjecture. Using Green and Tao's result and adopting the proof that the M\"obius function is a good weight for the
classical pointwise ergodic theorem \cite[Proposition 3.1]{AbKuLeDe} (see also \cite[Theorem C]{FanKon2019}), Eisner \cite[Theorem 2.2]{E2018} proved that \eqref{polynomial sar eq} holds for $f\in L^q(\nu)$ and for $\nu$-a.e. $x$ where $q>1$ and $\nu$ is any $T$-invariant measure. Fan \cite{Fan2019} proved such result for weight (other than M\"obius function) having Davenport's exponent $>1/2$.


As mentioned in \cite{E2018}, Nikos Frantzikinakis and Mariusz Lema\'{n}czyk observed that polynomial Sarnak conjecture is false without the minimality assumption. Let $\mathbf{a}=(a_n)_{n\in \mathbb{Z}}$ with 
\begin{equation}a_n:=\begin{cases} \mu(k), & n=k^2\text{ for some }k\in \mathbb{N} \\
 0 & \text{otherwise}
 \end{cases}.
 \end{equation}Then the sequence $\mathbf{a}$ is deterministic and $\lim_{N\rightarrow \infty}\frac{1}{N}\sum_{n=1}^N\mu(n)a_{n^2}=\frac{6}{\pi^2}$. One can see more details in \cite[Page 4]{E2018}. 
 
 In \cite{HLSY2019}, the authors showed that every sequence is `close to' a Toeplitz sequence in Besicovitch distance, whose entropy is controlled by the entropy of the original sequence. The proof in \cite{HLSY2019} is constructive. In the present paper, inspired by the construction in \cite{HLSY2019}, we will  find a counterexample of polynomial Sarnak conjecture for minimal systems.
Here is our main result.

 \begin{thm}\label{main theorem}
 Let $\epsilon>0$ be an arbitrary small number. 
 Then there exists a Toeplitz sequence $\mathbf{b}$ such that
 \begin{enumerate}
    \item the dynamical system $(X_{\mathbf{b}},\sigma)$ is of entropy zero;
    \item ${\displaystyle\limsup_{N\rightarrow \infty}\frac{1}{N}\left|\sum_{n=1}^Nb_{n^2}\mu(n)\right|>\frac{6}{\pi^2}-\epsilon.}$
\end{enumerate}
  \end{thm}

 A generalization of Theorem \ref{main theorem}, where M\"obius function is replaced by other arithmetic functions, is presented in Section \ref{sec:remark} (Theorem \ref{main theorem 2}). Theorem \ref{main theorem 2} is proven by the same argument as in the proof of Theorem \ref{main theorem}. Thus we focus on the proof of Theorem \ref{main theorem}.

  
 Since Toeplitz systems are minimal, Theorem \ref{main theorem} admits the following corollary \footnote{The equation \eqref{polynomial sar eq} fails for the Toeplitz system generated by $\mathbf{b}$  with $p(n)=n^2$, $x=\mathbf{b}$ and $f$ projection on the $0$-coordinate.}.
 \begin{cor}\label{main cor}
Polynomial Sarnak conjecture for minimal systems is false.
\end{cor} 
We point out that, simultaneously and independently of this work,  Adam Kanigowski, Mariusz Lema\'nczyk and Maksym Radziwi\l\l ~\cite{AdamLemMak} obtained an elegant alternative proof of Corollary \ref{main cor}. Their proof is based on the construction of their main counterexamples to prime number theorem in the class of continuous Anzai skew products. This Anzai skew product is uniquely ergodic, as well as minimal.

  
\section{Preliminiries}

In this section, we recall several basic notions in topological dynamical systems.

\subsection{Transitivity and minimality}

A topological system $(X, T)$ is said to be {\em transitive} if there exists some point
$x\in X$ whose orbit $\text{Orb}(x,T)=\{T^nx: n\in \mathbb{Z}\}$ is dense in $X$ and
we call such a point a {\em transitive point}. The system is said to be {\em
minimal} if the orbit of any point is dense in $X$.
A point $x\in X$ is called a {\em minimal point} if
$(\overline{\text{Orb}(x,T)}, T)$ is minimal.


\subsection{Symbolic dynamics}

Let $S$ be a finite alphabet with $k$ symbols, $k \ge 2$. We assume that $S=\{1,2,\cdots,k\}$. Let $\Sigma_k=S^{\mathbb{Z}}$ be the
space of all bi-infinite sequences $${\bf x}=\ldots x_{-1}x_0x_1 \ldots=(x_i)_{i\in \mathbb{Z}}, x_i \in
S$$
with the product topology. A metric on $\Sigma_k$ compatible with the topology is
given by $d({\bf x},{\bf y})=\frac{1}{1+m}$ for ${\bf x},{\bf y} \in \Sigma_k$, where $m=\min \{|n|:x_n \not= y_n
\}$. The shift map $\sigma: \Sigma_k \rightarrow \Sigma_k$
is the homeomorphism defined by $(\sigma {\bf x})_n = x_{n+1}$ for all $n \in \mathbb{Z}$. The
pair $(\Sigma_k,\sigma)$ is called a {\em full-shift dynamical system}. Any subsystem of $(\Sigma_k, \sigma)$
is called a {\em subshift system}.

Each element of $S^{\ast}= \bigcup_{k \ge 1} S^k$ is called {\em a word} or {\em a block} (over $S$).
We use $|A|=n$ to denote the length of $A$ if $A=a_1\ldots a_n$.
If $\omega=(\ldots \omega_{-1} \omega_0 \omega_1 \ldots) \in \Sigma_k$ and $a \le b \in \mathbb{Z}$, then
$\omega[a,b]=\omega_{a} \omega_{a+1} \ldots \omega_{b}$ is a $(b-a+1)$-word occurring in
$\omega$ starting at place $a$ and ending at place $b$.
Similarly we define $A[a,b]$ when $A$ is a word. A finite-length word $A$ {\em appears} in the word $B$ if there are some $a\le b$ such that
$B[a,b]=A$.

Let $(X,\sigma)$ be a subshift system. The cardinality of the collection of all $n$-words of $X$ is denoted by $p_X(n)$. Then the {\em topological entropy} of $(X,\sigma)$ is defined by
$$h(X,\sigma)=\lim_{n\to\infty}\frac{\log  p_X(n)}{n}.$$

Let ${\bf \omega}\in \Sigma_k$. Denote the orbit closure of ${\bf \omega}$ by $X_{\bf \omega}$. We call $h(X_{\bf \omega}, \sigma)$ the {\em entropy} of the sequence ${\bf \omega}$, and also denoted it by $h({\bf \omega})$.

\subsection{Odometers}

Fixing a sequence of positive integers $(n_i)_{i\in \mathbb{N}}$ with  $n_i\ge 2$, let $$X=\prod_{i=1}^\infty\{0,1,\ldots, n_i-1\}$$ with the discrete topology on each coordinate and the product topology on X. Let $T$ be the transformation that is the addition of $(1,0,0,\ldots)$ with carrying to the right. In other words, the transformation $T$ has the form
$$T(x_1,x_2,\ldots,x_k,x_{k+1},\ldots)=(0,0,\ldots,0,x_k+1,x_{k+1},\ldots)$$
where $k$ is the least entry such that $x_k<n_k-1$, and if there is no such $k$ then it produces $(0,0,\ldots)$. This system $(X,T)$ is called an {\em adding machine} or {\em odometer}. We write $(X,T)=\underleftarrow{\lim}(\mathbb{Z}_{n_i},T_{n_i})$, where $(\mathbb{Z}_{n_i},T_{n_i})$ is the $n_i$ element rotation. 

\subsection{Toeplitz systems}

Let $(X,T)$ be topological dynamical system. A point $x$ is {\em regularly recurrent} if for any open neighborhood $U$ of $x$, there is some $l_U\in \mathbb{N}$ such that $l_U\mathbb{Z}\subset \{n\in \mathbb{Z}:T^nx\in U\}$. It is clear that a regularly recurrent point is a minimal point. It is well-known that a topological dynamical system is a subsystem generated by a regularly recurrent point if and only if  it is a minimal almost one-to-one extension of an adding machine (see for instance \cite[Theorem 5.1]{downarowicz2005survey}).


A regularly recurrent point in a symbolic dynamics
is called a {\em Toeplitz sequence}. The subshift defined by a Toeplitz sequence ${\bf \omega}$ is called a {\em Toeplitz system}. Toeplitz systems can be characterized up
to topological conjugacy  with the following three properties (see for example \cite[Theorem 7.1 and Page 14]{downarowicz2005survey}):
\begin{itemize}
    \item [(1)] minimal,
    \item [(2)] almost 1-1 extensions of adding machines,
    \item [(3)] symbolic.
\end{itemize}

\section{Construction of the counter-example.}
In this section, we construct a counter-example to polynomial Sarnak conjecture. We first prove a technical arithmetic lemma which is crucial to the construction.

\subsection{A technical lemma.}

An integer $q$ is called a {\em quadratic residue modulo $n$} if 
there exists an integer $x$ such that
$$
x^2\equiv q \mod{n}.
$$
Throughout this section, we focus on prime power modulus. It is well-known that a nonzero number is a residue modulo $2^n$ with $n\ge 1$ if and only if it is of the form $4^r(8s+1)$ for some non-negative integers $r$ and $s$.

For $n\ge1$, we partition  $\{1,2, 3,\dots,2^n\}$ into
$$
B_{r,s,n}=\{1\leq k \leq 2^{n}: k^2\equiv 4^r(8s+1)\mod 2^n\},
$$
for $r,s\ge0$ and $4^r(8s+1)\le 2^n$.
\begin{lem}\label{number theory lemma}
Let $n\ge 1$. For any non-negative integers $r,s$ satisfying $4^r(8s+1)\le 2^n$, we have
$$\#B_{r,s,n}=2^r \#B_{0,s,n-2r}\leq 2^{r+2}.$$
\end{lem}
\begin{proof} 
Let $r,s$ be non-negative integers satisfying $4^r(8s+1)\le 2^n$. Clearly, we have $0\le r\leq \lfloor \frac{n}{2} \rfloor $, where $\lfloor x\rfloor$ is the greatest integer less than or equal to $x$. We firstly consider the case that $r\ge 1$. Let $k\in B_{r,s,n}$. It follows that \begin{equation*}
    \left(\frac{k}{2^r} \right)^2\equiv 8s+1 \mod 2^{n-2r},
\end{equation*}
implying that 
\begin{equation}\label{eq:tech 1}
    \frac{1}{2^r}B_{r,s,n}\cap\{1, 2, \dots, 2^{n-2r} \}=B_{0,s,n-2r}.
\end{equation}
We observe that 
\begin{equation*}
    1\le \frac{k}{2^r}\leq 2^{n-r}~\text{and}~ \left(\frac{k}{2^r}\right)^2\equiv \left(\frac{k}{2^r}+\ell\cdot2^{n-2r}\right)^2 \mod 2^{n-2r},
\end{equation*}
for all $0\le \ell\le 2^{r}-1$. This means that 
\begin{equation}\label{eq:tech 2}
    \begin{split}
        \frac{1}{2^r}B_{r,s,n}\cap&\{\ell\cdot2^{n-2r}+1, \ell\cdot2^{n-2r}+2, \dots, (\ell+1)\cdot 2^{n-2r} \} \\
&=\ell\cdot2^{n-2r}+\frac{1}{2^r}B_{r,s,n}\cap\{1, 2, \dots, 2^{n-2r} \}, 
    \end{split}
\end{equation}
for all $0\le \ell\le 2^{r}-1$. Combining \eqref{eq:tech 1} and \eqref{eq:tech 2}, we establish that 
$\#B_{r,s,n}=2^r\#B_{0,s,n-2r}.$ It remains to show that $\#B_{0,s,n}\leq 4$.

Let $a\in B_{0,s,n}$. Clearly, the integer $a$ is odd. It is easy to check that $\#B_{0,s,1}=1$, $\#B_{0,s,2}=2$ and $ \#B_{0,s,3}=4$ (in fact, $s$ has to be $0$ in these cases). Now we assume $n\ge 3$. Suppose $b\in B_{0,s,n}$ with $a^2\equiv b^2 \mod 2^n$. Without loss of generality, we might assume $b\ge a$. Let $m=b-a$ which is an even integer. By the fact that $a^2\equiv (a+m)^2\mod 2^n$, we get $(2a+m)m\equiv 0\mod 2^n$. Since $m$ is even, we have 
$$
\left(a+\frac{m}{2}\right)\frac{m}{2}\equiv 0 \mod 2^{n-2}.
$$
Since $a$ is odd, we see that the two integers $a+\frac{m}{2}$ and $\frac{m}{2}$ are exactly one odd number and one even number. It follows that $2^{n-2}$ divides one of the two integers $(a+\frac{m}{2})$ and $\frac{m}{2}$. Since $0\leq m<2^n$, we obtain that $m$ takes the value in one of $0$, $2^{n-1}$, $2^{n-1}-2a$ and $2^{n}-2a$. Therefore, we have 
$$
B_{0,s,n}\subset \{a,a+2^{n-1},2^{n-1}-2a,2^{n}-2a\}.
$$
This completes the proof.


\end{proof}
\subsection{Proof of main theorem.}\label{proof of main theorem}
Firstly we review the main result in \cite[Theorem 1.1]{HLSY2019} concerning the construction of a Toeplitz sequence ``close to" the given sequence. 
\begin{thm}\label{Toeplitz b}
Let $\mathbf{a}=(a_n)_{n\in \mathbb{Z}}\in \sum_k$. For any $\epsilon>0$, there exists a Toeplitz sequence $\mathbf{b}=(b_n)_{n\in \mathbb{Z}}$ such that 
\begin{enumerate}
    \item $h(X_{\mathbf{b}},\sigma)\leq 2h(X_{\mathbf{a}},\sigma)$;
    \item $\lim_{N\rightarrow \infty}\frac{1}{2N+1}\sum_{n=-N}^{N}|a_n-b_n|<\epsilon$.
\end{enumerate}
\end{thm}
 
 


Actually, the proof of Theorem \ref{Toeplitz b} is to construct a suitable sequence $\{\mathbf{a}^{(M)}\}_{M=1}^\infty$ (based on $\mathbf{a}$) by induction and then take ${\bf b}= (b_n)_{n\in \mathbb{Z}}= \lim_{M\to\infty} \mathbf{a}^{(M)}$.
The proof of our main theorem is influenced by this construction. Here is the rough outline of our proof. We set $\mathbf{a}$ to be the sequence given by Nikos  Frantzikinakis  and  Mariusz  Lema\'nczyk mentioned in Section \ref{sec:1}. We would like to find $\mathbf{b}$ of zero entropy so that not only  the Ces\`aro averages $\frac{1}{N}\sum_{n=1}^{N} |a_n-b_n|$ but also $\frac{1}{N}\sum_{n=1}^{N} |a_{n^2}-b_{n^2}|$ are sufficiently small. Since the sequence $\{n^2\}_{n=1}^{\infty}$ has zero density in $\mathbb{N}$, Theorem \ref{Toeplitz b} does not provide any information on the average $\frac{1}{N}\sum_{n=1}^{N} |a_{n^2}-b_{n^2}|$. Thus Theorem \ref{Toeplitz b} is not enough. 
By using Lemma \ref{number theory lemma} and adopting the method in the proof of Theorem \ref{Toeplitz b}, we carefully construct a specific sequence $\{\mathbf{a}^{(M)}\}_{M=1}^\infty$  and take ${\bf b}$ to be the limit of $ \mathbf{a}^{(M)}$. 

 \begin{proof}[Proof of Theorem \ref{main theorem}]
 Let $\epsilon>0$.
 Let $\mathbf{a}=(a_n)_{n\in \mathbb{Z}}$ be defined by 
\begin{equation}\label{value an}a_n:=\begin{cases} \mu(k), & n=k^2\text{ for some }k\in \mathbb{N} \\
 0 & \text{otherwise}.
 \end{cases}
 \end{equation}
We see from the proof of Lemma \ref{lem:1} that $h(X_{\mathbf{a}},\sigma)=0$.
 We would like to construct a Toeplitz sequence ${\bf b}= (b_n)_{n\in \mathbb{Z}}$ such that $h(X_{\mathbf{b}},\sigma)\leq 2h(X_{\mathbf{a}},\sigma)=0$ and
 the Ces\`aro averages of the sequence $|a_{n^2}-b_{n^2}|$ are less than $\epsilon$ along a sequence of increasing subsets of $\mathbb{N}$. Then we have the following 
 \begin{equation}\label{eq:goal}
      {\displaystyle\limsup_{N\rightarrow \infty}\frac{1}{N}\sum_{n=1}^Nb_{n^2}\mu(n)>\frac{6}{\pi^2}-\epsilon.}
 \end{equation} In fact, we will construct the sequence $\{\mathbf{a}^{(M)}\}_{M=1}^\infty$ by induction on $M$ and take ${\bf b}= (b_n)_{n\in \mathbb{Z}}$ as the limit of $\mathbf{a}^{(M)}$.

Now we begin to construct the sequence $\{\mathbf{a}^{(M)}\}_{M=1}^\infty$ by induction on $M$. Choose a decreasing sequence $(\epsilon_i)_{i\in \mathbb{N}}$ of positive numbers such that $\epsilon_0=\epsilon$ and $\epsilon_{i+1}<\epsilon_i/2$. Pick an increasing sequence $(n_i)_{i\in \mathbb{N}}$ of non-negative integers such that $n_0=0$ and $n_i>\max\{ 2n_{i-1}, n_{i-1}+2-\log_2 \epsilon_i\} $. 

\medskip

\noindent \textbf{Step $1$}: Let $l_1=2^{n_1}\in \mathbb{N}$. Then $\frac{4}{l_1}\leq \epsilon_1$. Note that $a_0=0$. Let ${\bf a^{(1)}}=\{a_n^{(1)}\}_{n\in \mathbb{Z}}$ be the sequence defined as follows:
$$a_n^{(1)}=\left\{
              \begin{array}{ll}
                0, & \hbox{if $n\in l_1\mathbb{Z}$;} \\
                a_n, & \hbox{otherwise.}
              \end{array}
            \right.
$$

\medskip

Assume that the sequences $\mathbf{a}^{(i)}$ are well defined for all $i
\le M-1$. Now we construct $\mathbf{a}^{(M)}$.

\medskip

\noindent \textbf{Step $M$:}  Let $l_M=2^{n_M}\in \mathbb{N}$. Then $n_M > 2n_{M-1}$ and 
\begin{equation}\label{nM and nM-1}
    \frac{4l_{M-1}}{l_M}=4\cdot 2^{n_{M-1}-n_M}\leq \epsilon_M.
\end{equation} 
Define $\mathbf{a}^{(M)}\in \{-1,0,1\}^{\mathbb{Z}}$  as follows:

\begin{equation}\label{equation-construction}
   \left\{
              \begin{array}{ll}
                \mathbf{a}^{(M)}[rl_M-l_{M-1}, rl_M+l_{M-1}-1]=\mathbf{a}^{(M-1)}[-l_{M-1}, l_{M-1}-1], & \hbox{for all $r\in \mathbb{Z}$;} \\
                a_n^{(M)}=a_n^{(M-1)}, & \text{otherwise}.
              \end{array}
            \right.
\end{equation}


\medskip

By induction on $M$, the sequences  $\mathbf{a}^{(M)}$ are well established for all $M\in \mathbb{N}$. It is clear that the limit ${\bf b}= \lim_{M\to\infty} \mathbf{a}^{(M)}$ exists and that $\mathbf{b}$ is a Toeplitz sequence. 
By Lemma \ref{lem:1}, the dynamical system $(X_{\mathbf{b}},\sigma)$ is of entropy zero. 
Now we only need to show \eqref{eq:goal}.
 

Let $D=\{n\in \mathbb{Z}:a_n\neq b_n\}$ and $$D_M=\{n\in \mathbb{Z}:a_n^{(M-1)}\neq a_n^{(M)}\},~\text{for}~M\ge 1.$$
Since ${\bf b}$ is the limit of $\mathbf{a}^{(M)}$, it is clear that $D\subset \bigcup_{M=1}^{\infty}D_M$. 
It is easy to see that $\min_{n\in D_M}|n|\geq 2^{n_M}-2^{n_{M-1}}$. 
We denote by $K_M$ the largest integer that is smaller than $(2^{n_M}-2^{n_{M-1}})^{1/2}$ and divided by $2^{n_{M-1}}$.
Let $A_M=\{k^2:1\leq k\leq  K_M\}$. Clearly, $\# A_M=K_M$.
By Lemma \ref{lem:2}, we have that
\begin{equation}\label{number of sets 0}\frac{\#\left(A_M\cap (\bigcup_{m=1}^{M-1}D_m)\right)}{K_M}<\frac{1}{2}\sum_{m=1}^{M-1}\epsilon_m,
\end{equation}
for all $M\geq 1$. 
By the fact that $a_n,b_n\in \{-1,0,1\}$ and the definitions of $D$ and $K_M$, we have that 
\begin{equation}\label{an sim bn}|\frac{1}{K_M}\sum_{n=1}^{K_M}(a_{n^2}-b_{n^2})\mu(n)|\leq 2\frac{\#\left(A_M\cap D\right)}{K_M}, ~\text{for all}~M\ge 1.
\end{equation}
As $\min_{n\in D_M}|n|\geq 2^{n_M}-2^{n_{M-1}}$, we see that $A_M\cap D_{M'}=\emptyset$ for any $M'\geq M$. It follows that 
\begin{equation}\label{eq:main thm 000}
    \#\left(A_M\cap D\right)\leq \#\left(A_M\cap (\bigcup_{m=1}^{\infty}D_m)\right)=\#\left(A_M\cap (\bigcup_{m=1}^{M-1}D_m)\right).
\end{equation}
Combining \eqref{number of sets 0}, \eqref{an sim bn} and \eqref{eq:main thm 000}, we obtain that  $$|\frac{1}{K_M}\sum_{n=1}^{K_M}(a_{n^2}-b_{n^2})\mu(n)|\leq2\frac{\#\left(A_M\cap D\right)}{K_M}\leq\frac{\#\left(A_M\cap (\bigcup_{m=1}^{M-1}D_m)\right)}{K_M}\leq \sum_{m=1}^{M-1}\epsilon_m,$$
for all $M\geq 1$. As ${\displaystyle\sum_{m=1}^\infty\epsilon_m<\epsilon}$, we have that
\begin{equation}\label{eq:main thm 0}
    |\frac{1}{K_M}\sum_{n=1}^{K_M}(a_{n^2}-b_{n^2})\mu(n)|\leq \epsilon, ~\text{for all}~M\ge 1.
\end{equation}
On the other hand, it is easy to compute that $${\displaystyle\lim_{M\rightarrow \infty}\frac{1}{K_M}\sum_{n=1}^{K_M}a_{n^2}\mu(n)=\lim_{N\rightarrow \infty}\frac{1}{N}\sum_{n=1}^Na_{n^2}\mu(n)=\lim_{N\rightarrow \infty}\frac{1}{N}\sum_{n=1}^N\mu(n)^2}=\frac{6}{\pi^2}.$$ 
Together with \eqref{eq:main thm 0}, we obtain that 
\begin{equation}
    \begin{split}
        &\limsup_{N\rightarrow \infty}\frac{1}{N}\sum_{n=1}^Nb_{n^2}\mu(n) \\ &\geq \limsup_{M\rightarrow \infty}\frac{1}{K_M}\sum_{n=1}^{K_M}b_{n^2}\mu(n)\\
&\ge  \limsup_{M\rightarrow \infty}\frac{1}{K_M}\left|\sum_{n=1}^{K_M}a_{n^2}\mu(n)\right|-\limsup_{M\rightarrow \infty}\frac{1}{K_M}\left|\sum_{n=1}^{K_M}(a_{n^2}-b_{n^2})\mu(n)\right|\\
&>\frac{6}{\pi^2}-\epsilon.
    \end{split}
\end{equation}
\end{proof}

\begin{lem}\label{lem:1}
The dynamical system $(X_{\mathbf{b}}, \sigma)$ is strictly ergodic and measure-theoretically conjugate to
an odometer w.r.t. the Haar measure. In particular, $(X_{\mathbf{b}}, \sigma)$ is of zero entropy.
\end{lem}
\begin{rem*}
By the consturction and the methods in \cite{HLSY2019}, one can deduce that 
\begin{equation}\label{entropy bounded}
p_{\mathbf{b}}(l_M)\le (l_M+1)p_{\mathbf{a}}(l_M)^2,
\end{equation}
for all $M\ge 1$. Therefore, the dynamical system $(X_{\mathbf{b}}, \sigma)$ is of zero entropy. On the other hand, as the set $\{n\in\mathbb{Z}:a_n\neq 0\}$ has the upper Banach density zero, one can also prove the zero entropy of $(X_{\mathbf{b}}, \sigma)$ by the properties of Toeplitz sequence and the mean-continuous methods (see \cite{downarowicz2005survey,DG2016} for more details).
\end{rem*}
\begin{proof}
By the proof of \cite[Theorem 5.1]{downarowicz2005survey}, there is a sequence of invariant clopen subspaces $\{X_M\}_{M\ge 1}$ such that $\mathbf{b}\in X_M$ and we have a periodic cycle of disjoint sets  
$$
X_M \to \sigma(X_M) \to \cdots \to \sigma^{l_M-1}(X_M) \to X_M,
$$
for each $M\ge 1$.
We define a factor map $\pi:(X_{\mathbf{b}}, \sigma)\rightarrow \underleftarrow{\lim}(\mathbb{Z}_{l_{M}},T_{l_{M}})$ by 
$$\mathbf{x}\mapsto (j_1, j_2, \ldots,j_M,\ldots) ~\text{if}~\mathbf{x}\in \sigma^{j_M}(X_M).$$
Now we fix an $(j_1,\ldots,j_M,\ldots)\in \underleftarrow{\lim}(\mathbb{Z}_{l_{M}},T_{l_{M}})$ and pick arbitrarily $\mathbf{x}, \mathbf{y}\in \pi^{-1}(j_1,\ldots,j_M,\ldots)$. We will calculate the upper density of $B:=\{n\in \mathbb{Z}:x_n\neq y_n\}$. For arbitrary $N\geq 1$, let $l_{M-1} \leq N<l_{M}$ for some $M$.  Assume that $T^{p_n}\mathbf{b}\rightarrow \mathbf{x}$ and $T^{q_n}\mathbf{b}\rightarrow \mathbf{y}$. Since $X_M$ is clopen, there exists an $N_M$ such that for $n>N_M$, $T^{p_n-j_M}\mathbf{b}, T^{q_n-j_M}\mathbf{b}\in X_M$, i.e. $l_M|p_n-j_M $, $l_M|p_n-j_M $ for $n\geq N_M$. It follows that $p_n\equiv q_n(\mod l_{M})$ for $n\geq N_M$. On the other hand, we can find an $n>N_M$ such that $(x_0,\ldots,x_{N-1})=(b_{p_n},\ldots,b_{p_n+N-1})$ and $(y_0,\ldots,y_{N-1})=(b_{q_n},\ldots,b_{q_n+N-1})$. 
Define the set $$B_p:=\{k\in [p_n,p_n+N-1]:b_k\neq 0\text{ and }k\notin [rl_{\tilde{M}}-l_{\tilde{M}-1}, rl_{\tilde{M}}+l_{\tilde{M}-1}-1]~\forall r\in \mathbb{Z},\tilde{M}\leq M\}.$$ The set $B_q$ is defined similarly. As $p_n\equiv q_n(\mod l_{M})$, by the construction of $\mathbf{b}$ one has that
\begin{equation}\label{eq:n1}
    \#(B\cap [1,N])\leq \#(B_p)+\# (B_q).
\end{equation}
Let $A:=\{n\in\mathbb{Z}:a_n\not=0\}$. Notice that there exist $l_M\le p_n'<p_n''\le l_{\hat{M}}$ with $p_n''-p_n'\le N+1$ for some $\hat{M}>M$ such that for any $k\in B_p$ either $b_k=a_k$ or $b_k=a_r$ for $p_n'\le r\le p_n''$ with $l_{\hat{M}}|k-r$.
Thus 
\begin{equation}\label{eq:n2}
    \#(B_p)\le \# (A\cap [p_n,p_n+N-1])+ \#(A\cap [p_n',p_n'']).
\end{equation}
Combing \eqref{eq:n1}, \eqref{eq:n2} and the fact that $A$ has the upper density zero, we obtain
\begin{equation}\label{upper density of B}\overline{d}(B)=\limsup_{N\rightarrow \infty}\frac{1}{N}\#(B\cap [1,N])\leq 4\limsup_{M,N\rightarrow \infty}\frac{1}{N}\#(A\cap [M+1,M+N])=0.
\end{equation}
It follows from \eqref{upper density of B} that $$\lim_{N\rightarrow \infty}\frac{1}{N}d(\sigma^nx,\sigma^ny)=0,$$ 
for any metric $d$ on $X_{\mathbf{b}}$ compatible with its topology.
Then by the proof 1.$\Rightarrow$ 3. of \cite[Theorem 2.1]{DG2016}, we deduce that
 $(X_{\mathbf{b}}, \sigma)$ is strictly ergodic and measure-theoretically conjugate to
the odometer $\underleftarrow{\lim}(\mathbb{Z}_{l_{M}},T_{l_{M}})$.

\end{proof}

\begin{lem}\label{lem:2}
For any $M\geq 1$, one has that 
\begin{equation}\label{number of sets}\frac{\#\left(A_M\cap (\bigcup_{m=1}^{M-1}D_m)\right)}{K_M}<\frac{1}{2}\sum_{m=1}^{M-1}\epsilon_m.
\end{equation}
\end{lem}
\begin{proof}
Note that every positive integer $n$ can be uniquely expressed by the finite sum
$$n=\sum_{M=0}^\infty h_M(n)2^{n_M},$$
with the coefficients $\ 0\leq h_M(n)\le2^{n_{M+1}-n_M}-1$. For $1\leq m\leq M-1$, we define 
$$D_{m}^{(1)}=\{n\in \mathbb{Z}:  h_{m-1}(n)=0\}$$ and $$D_{m}^{(2)}=\{n\in \mathbb{Z}:h_{m-1}(n)=2^{n_{m}-n_{m-1}}-1\}.$$
Obviously, $D_m\subset D_{m}^{(1)}\cup D_{m}^{(2)}$. Noticing that $2^{n_m}$ divides the number $K_M=\#A_M$ and $ (k+2^{n_m})^2$ is congruent  to $k^2$ modulo $ 2^{n_m}$ for any integer $k$, we obtain that the set $A_M\cap D_{m}^{(i)}$ is ``periodic" with ``period" $2^{n_m}$ in the following sense:
\begin{equation}\label{eq:main thm 2}
    A_M\cap D_{m}^{(i)}=\bigsqcup_{0\leq \ell\leq \frac{K_M}{2^{n_m}}-1}\{(k+\ell 2^{n_m})^2: k\in A_M\cap D_{m}^{(i)}, k\le 2^{n_m} \},
\end{equation}
for $i=1,2$. Thus it is sufficient to focus on $A_M\cap D_{m}^{(i)}\cap\{k^2: 1\le k\le 2^{n_m} \}$.


Notice that for $k\in B_{r,s,n_m}$, we have $h_{m-1}(k^2)=h_{m-1}(4^r(8s+1))$. It follows that the set $\{k^2: k\in B_{r,s,n_m}\}$ is contained entirely in either $D_{m}^{(1)}\cap\{k^2: 1\le k\le 2^{n_m} \}$ or $D_{m}^{(2)}\cap\{k^2: 1\le k\le 2^{n_m} \}$ for every $r,s$.
Thus we have the decomposition that
\begin{equation}\label{eq:main thm 1}
    A_M\cap D_{m}^{(i)}\cap\{k^2: 1\le k\le 2^{n_m} \}=\bigsqcup_{r,s} \left(A_M\cap B_{r,s,n_m}\right), ~\text{for}~i=1,2,
\end{equation}
where $r,s$ run over all non-negative integers satisfying that $ 4^r(8s+1)\le 2^{n_m}$ and that $4^r(8s+1)\in D_{m}^{(i)}$.
Combining \eqref{eq:main thm 2} and \eqref{eq:main thm 1}, we have
\begin{equation}\label{eq:main thm 3}
        \#(A_M\cap D_{m}^{(i)})\leq {\displaystyle\frac{K_M}{2^{n_m}}\sum_{r,s} \# B_{r,s,n_m}}
        \le {\displaystyle \frac{K_M}{2^{n_m}}\sum_{r,s}2^{r+2}}, ~\text{for}~i=1,2,
\end{equation}
where $r,s$ run over all non-negative integers satisfying that $ 4^r(8s+1)\le 2^{n_m}$ and $4^r(8s+1)\in D_{m}^{(i)}$. In fact, the last inequality is due to Lemma \ref{number theory lemma}. 
We calculate the upper bound of $\#(A_M\cap D_{m}^{(1)})$ as follows: by \eqref{eq:main thm 3} and the definitions of $n_M$ and $\epsilon_M$, we have
\begin{equation*}
    \begin{split}
    \#(A_M\cap D_{m}^{(1)})&\le
    \frac{4K_M}{2^{n_m}} \sum_{r=0}^{\lfloor n_{m-1}/2\rfloor} 2^r \sum_{s} \# \{s\ge 0: 4^r(8s+1)<2^{n_m-1} \}\\
    &\le \frac{4K_M}{2^{n_m}}\sum_{r=0}^{\lfloor n_{m-1}/2\rfloor}\frac{2^{n_{m-1}-r}}{8}=\frac{K_M}{2^{n_m-n_{m-1}}}\sum_{r=0}^{\lfloor n_{m-1}/2\rfloor}\frac{1}{2^{r+1}}
    \\ &\leq \frac{K_M}{2^{n_m-n_{m-1}}}<\frac{\epsilon_m}{4}K_M,
    \end{split}
\end{equation*}
where the first inequality is due to the fact that if two non-negative integers $r,s$ satisfy $ 4^r(8s+1)\le 2^{n_m}$ and $4^r(8s+1)\in D_{m}^{(1)}$, then $1\leq 4^r(8s+1)< 2^{n_{m-1}}$.
On the other hand, noticing that if two non-negative integers $r,s$ satisfy $ 4^r(8s+1)\le 2^{n_m}$ and $4^r(8s+1)\in D_{m}^{(2)}$, then $2^{n_m}-2^{n_{m-1}}\leq 4^r(8s+1)< 2^{n_{m}}$, we obtain the upper bound of $\#(A_M\cap D_{m}^{(2)})$ in a similar way with that of $\#(A_M\cap D_{m}^{(1)})$:
\begin{equation*}
    \begin{split}
    &\#(A_M\cap D_{m}^{(2)})\\
    &\le\frac{4K_M}{2^{n_m}} \sum_{r=0}^{\lfloor n_{m-1}/2\rfloor} 2^r \sum_{s} \# \{s\ge 0: 2^{n_m}-2^{n_{m-1}}\leq 4^r(8s+1)< 2^{n_{m}} \}\\
   &<\frac{\epsilon_m}{4}K_M.
    \end{split}
\end{equation*}


Therefore, we have
\begin{equation*}
    \begin{split}
        \#\left(A_M\cap (\bigcup_{m=1}^{M-1}D_m)\right)
&\leq \#\left(A_M\cap (\bigcup_{m=1}^{M-1}(D_{m}^{(1)}\cup D_{m}^{(2)}))\right)\\
&\leq {\displaystyle\sum_{\substack{i=1,2\\ 1\leq m\leq M-1}}\#(A_M\cap D_{m}^{(i)})}\\
&< \frac{K_M}{2}\sum_{m=1}^{M-1}\epsilon_{m-1},
    \end{split}
\end{equation*}
which completes the proof.
\end{proof}





\section{General case.}\label{sec:remark}


In this section, we discuss a generalization of Theorem \ref{main theorem}. 
Firstly applying \cite[Theorem 4.1]{downarowicz2015odometers} to Lemma \ref{lem:1}, we obtain that the sequence $\mathbf{b}$ in Section \ref{proof of main theorem} satisfies Sarnak Conjecture.
Then note that we can adapt the proof of Theorem \ref{main theorem} to arithmetic (not necessarily multiplicative) function other than M\"obius function:
 \begin{thm}\label{main theorem 2}
 Let $\omega: \mathbb{N} \to \mathbb{C}$.  
 Suppose $\limsup_{n\to \infty} \frac{1}{N}\sum_{n=1}^{N}|\omega(n)|>0$.
 Then there exists a Toeplitz sequence $\mathbf{b}$ such that
 \begin{enumerate}
    \item the dynamical system $(X_{\mathbf{b}},\sigma)$ is of entropy zero and strictly ergodic;
    \item the sequence $\mathbf{b}$ satisfies Sarnak Conjecture;
    \item ${\displaystyle\limsup_{N\rightarrow \infty}\frac{1}{N}\left|\sum_{n=1}^Nb_{n^2}\omega(n)\right|>0.}$
\end{enumerate}
  \end{thm}
The proof of $1.$ and $3.$ of Theorem \ref{main theorem 2} is followed by the same argument of the proof of Theorem \ref{main theorem}, where we replace $\mathbf{a}$ by the sequence $\mathbf{a}(\omega)=(a_n(\omega))_{n\in \mathbb{N}}$ defined by
\begin{equation}
a_n(\omega):=\begin{cases} \omega(k), & n=k^2\text{ for some }k\in \mathbb{N} \\
 0 & \text{otherwise}.
 \end{cases}
 \end{equation}
The proof of $2.$ of Theorem \ref{main theorem 2} is followed by Lemma \ref{lem:1} and \cite[Theorem 4.1]{downarowicz2015odometers}.

\section{Further remarks.}
Recall that the {\it sequence topological entropy} of a dynamical system $(X, T)$ along the sequence $(P(k))_{k\in \mathbb{N}}$ is defined by
\begin{equation}\label{eq:seq}
    h_P(X, T)=\sup_{\alpha}\limsup_{n\to \infty} \frac{\log \mathcal{N}(\bigvee_{0\le k\le n-1}T^{-P(k)}\alpha)}{n},
\end{equation}
where $\alpha$ runs over all finite open covering of $X$ and $\mathcal{N}(\theta)$ is the minimal cardinality among all cardinalities of sub-covers of $\theta$.
A dynamical system is said to be of {\em polynomial entropy zero} if its sequence topological entropy is equal to zero along every polynomial sequence.

In the same paper where polynomial Sarnak conjecture was formulated, Eisner \cite{E2018} also proposed a weaker version of  polynomial Sarnak conjecture for minimal system:
\begin{ques}[Conjecture 2.4, \cite{E2018}]\label{Polynomial Sarnak 2}
Let $(X,T)$ be a minimal topological dynamical system of zero polynomial entropy. Then \eqref{polynomial sar eq} holds for every $f\in C(X)$, every polynomial $p:\mathbb{N}\rightarrow \mathbb{N}_0$ and every $x\in X$.
\end{ques}

Since a dynamical system of zero sequential entropy (along some increasing sequence in $\mathbb{N}$) is of zero topological entropy, Conjecture \ref{Polynomial Sarnak 2} follows from by polynomial Sarnak conjecture.

Unfortunately, We don't know how to calculate the polynomial entropy of $\mathbf{b}$ (in Theorem \ref{main theorem}) along any polynomial so far. Thus we have no idea whether the sequence $\mathbf{b}$ is a counter-example of Conjecture  \ref{Polynomial Sarnak 2}. However, we could show the following property of $\mathbf{b}$ which is related to the  polynomial entropy of $\mathbf{b}$. Let $P: \mathbb{N}\rightarrow \mathbb{N}_0$ a polynomial. For every $n\in \mathbb{N}$, we define $\beta_P(n)$ to be the largest integer $m$ such that $P(m)\le n$.  By \eqref{entropy bounded}, we have
\begin{equation*}
\begin{split}
    \limsup_{M\to \infty} \frac{\log p_{\mathbf{b}}(l_M)}{\beta_P(l_M)} &= \limsup_{M\to \infty} \frac{\log (l_M+1)Q(l_M)^2}{\beta_P(l_M)}\\
    &\le \lim_{n\to \infty} \frac{\log (P(n)+1)Q(P(n))^2}{n}=0.
\end{split}
\end{equation*}
We remark that the above lim-sup  is incomparable with the limit \eqref{eq:seq}. However, the above lim-sup is the upper bound of the limit that has the form \eqref{eq:seq} where $\limsup$ is replaced by $\liminf$.

\section*{Acknowledgement.}
We are grateful to Tanja Eisner and Mariusz Lema\'nczyk for helpful remarks and encouraging discussions. Thanks are also due to Ai-Hua Fan for valuable comments on an earlier draft. We thank the reviewer's valuable remarks and comments.

\specialsectioning
\bibliographystyle{alpha}
\bibliography{references}

\Addresses
\end{document}